\documentclass[reqno,12pt]{amsart}
\usepackage{amsmath,amsthm,enumerate,cite,graphics,pstricks,amsfonts,latexsym,amsopn,verbatim,amscd,amssymb}
\usepackage{color}
\usepackage{psfrag}
\usepackage[all]{xy}

\theoremstyle{plain}
\newtheorem{thm}{Theorem}[section]

\newtheorem{lem}[thm]{Lemma}
\newtheorem{cor}[thm]{Corollary}
\newtheorem{prop}[thm]{Proposition}

\theoremstyle{definition}

\newtheorem{rem}[thm]{Remark}

\DeclareMathOperator{\Aut}{Aut}

\DeclareMathOperator{\Cent}{Cent}

\begin{document} 

\title[On capable groups  of order $p^2q$]{On capable groups  of order $p^2q$} 

\author[S. J. Baishya  ]{Sekhar Jyoti Baishya} 
\address{S. J. Baishya, Department of Mathematics, Pandit Deendayal Upadhyaya Adarsha Mahavidyalaya, Behali, Biswanath-784184, Assam, India.}

\email{sekharnehu@yahoo.com}

\begin{abstract}
A group is said to be capable if it is the central factor of some group. In this paper, among other results we have characterized capable groups of order $p^2q$, for any distinct primes $p, q$, which extends Theorem 1.2 of S. Rashid, N. H. Sarmin, A. Erfanian, and N. M. Mohd Ali, {\em On the non abelian tensor square and capability of groups of order $p^2q$}, Arch. Math., \textbf{97}  (2011), 299--306. We have also computed the number of distinct element centralizers of a group (finite or infinite) with central factor of order $p^3$, which extends Proposition 2.10 of S. M. Jafarian Amiri, H. Madadi and H. Rostami,  {\em On $F$-groups with the central factor of order $p^4$},  Math. Slovaca, \textbf{67} (5) (2017), 1147--1154.  
\end{abstract}

\maketitle

\section{Introduction} \label{S:intro}

Given a group $G$, let $\Cent(G)$ denote the set of centralizers of $G$, i.e., $\Cent(G)=\lbrace C(x) \mid x \in G\rbrace $, where $C(x)$ is the centralizer of the element $x$ in $G$. The study of finite groups in terms of $|\Cent(G)|$, becomes an interesting research topic in last few years. It is easy to see that $|\Cent(G)|=1$ if and only if $G$ is abelian. Belcastro and Sherman proved that there is no group with $|\Cent(G)|=2$ or $3$ and   that for a finite group $|\Cent(G)|=4$ if and only if $\frac{G}{Z(G)} \cong C_2 \times C_2$ and  $|\Cent(G)|=5$ if and only if $\frac{G}{Z(G)} \cong S_3$ or $C_3 \times C_3$ (\cite{ctc092}). After that the structure of all finite groups are determined for $|\Cent(G)|=6, 7, 8, 9, 10$ in terms of $\frac{G}{Z(G)}$ (see \cite{ctc09, ed09, baishya1, amiri2019, amiri20191}). There are several other interesting results concerning the impact of $|\Cent(G)|$. For example, every finite group with $|\Cent(G)|\leq 21$ is solvable (\cite{zarrin0942}). Again, if $G$ is a solvable group (not necessarily finite) of derived length $d$, then $d \leq |\Cent(G)|$ (\cite{zarrin0941}). More recently, we have if $G$ is a non-trivial group such that $|\Cent(G)| \geq \frac{2\mid G \mid}{3}$, then $G \cong S_3, S_3 \times S_3$ or $D_{10}$ (\cite {rostami}) etc. More information on this and related concepts may be found in  \cite{amiri17, en09, ctc091, ctc099, baishya, zarrin094}.

On the other hand a group $G$ is said to be capable if there exists a group $H$ such that $G \cong \frac{H}{Z(H)}$. The study of capable groups was initiated by R. Baer \cite{baer}, who determined all capable groups which are direct sums of cyclic groups. Consequently, all capable finite abelian groups are characterised by R. Baer's result. For finite $p$-groups capability is closely related to their classification. The authors in \cite{capable} characterised capable extraspecial groups (only $D_8$ and the extraspecial groups of order $p^3$ and exponent $p$ are capable); they also studied the metacyclic capable groups. 

While studying the number of distinct element centralizers in a group $G$ it is observed that 
$\mid \Cent(G) \mid$ has strong influence on the group and $\mid \Cent(G) \mid$ depends on the structure of $\frac{G}{Z(G)}$. It is also of independent interest to obtain some information about the group $G$ from a given $\frac{G}{Z(G)}$. In this paper, for any primes $p, q, r$ (not necessarily distinct) we have studied capable groups of order $pqr$ and computed $\mid \Cent(G) \mid$ for finite groups $G$ with  $\mid \frac{G}{Z(G)}\mid =pqr$.  

In this paper, all groups are finite (however  Lemma \ref{rem144}, Remark \ref{rem1} and Proposition \ref{lem195} holds for any group) and all notations are usual. For example $G'$, $Z(G)$ denotes the commutator subgroup  and the center of a group $G$ respectively, $C_n$ denotes the cyclic group of order $n$,  $D_{2n}$ denotes the dihedral group of order $2n$, $A_4$ denotes the alternating group of degree $4$,  and $ C_n {\rtimes}_\theta C_p$ denotes semidirect product of $C_n$ and $C_p$, where $\theta : C_p \longrightarrow \Aut(C_n)$ is a homomorphism.

 \section{The main results}

In this section, we prove the main results of the paper. However, we begin with the following lemma.

\begin{lem}\label{rem144}
 Let $G$ be any group. If $\mid \frac{G}{Z(G)} \mid=pqr$ where $p, q, r$ are primes (not necessarily distinct), then $C(x)$ is abelian for any $x \in G \setminus Z(G)$. 
\end{lem}

\begin{proof}
Let $x \in G \setminus Z(G)$. If $\frac{C(x)}{Z(G)}$ is cyclic, then $C(x)$ is abelian. Now, suppose $\mid \frac{C(x)}{Z(G)} \mid=pq$. Then $o(xZ(G))=p, q$ or $pq$. If $o(xZ(G))=pq$, then $C(x)$ is abelian. Next suppose $o(xZ(G))\neq pq$. Then there exists some $y \in C(x)$ such that 
$\frac{C(x)}{Z(G)}= \langle xZ(G), yZ(G) \rangle$. Consequently, $C(x)=\langle x, y, Z(G) \rangle$ and hence $C(x)$ is abelian. If $\mid \frac{C(x)}{Z(G)} \mid=pr$ or $qr$, then using similar arguments we can show that $C(x)$ is abelian.
\end{proof}

\begin{rem}\label{rem1}
Recall that a group $G$ is said to be a CA-group if $C(x)$ is abelian for any $x \in G \setminus Z(G)$. It is easy to see that for such groups $C(x) \cap C(y)=Z(G)$ for any two distinct proper centralizers $C(x)$ and $C(y)$. Also we have seen that if $\mid \frac{G}{Z(G)} \mid=pqr$ where $p, q, r$ are primes (not necessarily distinct), then $G$ is a CA-group. 
\end{rem}

We now consider the groups of order $pqr$ where $p<q<r$.

\begin{prop}\label{lem19}
If $G$ is a  non-abelian  group of order $pqr, p<q<r$ are primes, then $\mid \Cent(G) \mid= q+2, r+2$ or $qr+2$. 
\end{prop}

\begin{proof}
If $\mid Z(G) \mid \neq 1$, then by \cite[Corolary 2.5 and Proposition 2.8]{baishya}, we have $\mid \Cent(G) \mid = \mid G'\mid+2 = q+2$ or $r+2$. On the otherhand if $\mid Z(G) \mid = 1$, then by Remark \ref{rem1}, $C(x)$ is cyclic for any $x \in G \setminus Z(G)$. Moreover, $G$ has a normal subgroup $N$ of order $qr$ and $G'$ is cyclic (since $G$ is metacyclic). Hence $\frac{G}{N}$ is cyclic of order $p$ and so $G' \subseteq N$. Consequently, we have $\mid G' \mid=q, r$ or $qr$. 

Now suppose $\mid G' \mid=q$. Then $N$ is an abelian normal subgroup of $G$ of  index $p$ and so by \cite[Lemma 4 (p. 303)]{zumud}, we have $\mid G \mid =p\mid G'\mid \mid Z(G)\mid$, which is a contradiction.

Next suppose $\mid G' \mid=r$. If $G' \subsetneq C(x)$ for some $x \in G' \setminus \lbrace1\rbrace$, then $C(x)$ is an abelian normal subgroup of $G$ of prime index and hence $\mid Z(G) \mid \neq 1$ by \cite[Lemma 4 (p. 303)]{zumud}, which is a contradiction. On the otherhand, if $C(x)=G'$ for all $ x\in G' \setminus \lbrace1\rbrace$, then by \cite[Proposition 1.2.4]{elizabeth}, we have $G$ is a Frobenius group with Frobenius kernel $G'$ and cyclic Frobenius complement of order $pq$. Consequently $\mid \Cent(G) \mid= r+2$.

Finally, suppose $\mid G' \mid=qr$. Then $G'$ is an abelian normal subgroup of $G$ of prime index and so by \cite[Theorem 2.3]{baishya}, we have $\mid \Cent(G) \mid= qr+2$. 

\end{proof}

As an immediate consequence we have the following corollary:

\begin{cor}\label{lem12}
If $G$ is a finite non-abelian group of order $pqr, p<q<r$  being primes, then $\mid \Cent(G) \mid= \mid G' \mid+2$. 
\end{cor}

 Recall that a group $G$ is said to be capable if there exists a group $H$ such that $G \cong \frac{H}{Z(H)}$. The following result on capable groups of order $pqr$ where  $p<q<r$ will be used in the next Proposition.

\begin{prop}\label{lem2191}
$G$ is a capable group of order $pqr$ ($p<q<r$ being primes) if and only if $\mid Z(G) \mid =1$.   
\end{prop}

\begin{proof}
Suppose $G$ is a capable group of order $pqr$, $p<q<r$ being primes. Then there exists a group $H$ such that $G \cong \frac{H}{Z(H)}$. Now suppose $\mid Z(\frac{H}{Z(H)})\mid=p$. In view of Remark \ref{rem1}, 
$ \frac{H}{Z(H)}$ has a cyclic centralizer say $C(xZ(H))=\langle sZ(H)\rangle $ of order $pq$ and a cyclic centralizer say $C(yZ(H))=\langle tZ(H)\rangle $ of order $pr$. But then  $\mid \frac{H}{C(s)} \mid=r$ and $\mid \frac{H}{C(t)} \mid=q$ and using Remark \ref{rem1} again, we have $\mid \frac{H}{Z(H)}\mid =\mid \frac{H}{{C(s) \cap C(t)}}\mid \leq qr$, which is a contradiction. Using similar arguments, we can show that  $\mid Z(\frac{H}{Z(H)})\mid \neq q, r$. Hence $\mid Z(G) \mid =1$. Converse is trivial.
\end{proof}

Now we consider the case where $\frac{G}{Z(G)}$ be of order $pqr, p<q<r$ being primes.

\begin{thm}\label{lem191}
If $G$ is a finite group such that $\mid \frac{G}{Z(G)} \mid=pqr$,   $p<q<r$ being primes, then $\mid \Cent(G) \mid= r+2$ or $qr+2$. 
\end{thm}

\begin{proof}

By  Proposition  \ref{lem2191}, we have $\mid Z(\frac{G}{Z(G)})\mid =1$. Again, since $\frac{G}{Z(G)}$ is metacyclic therefore $(\frac{G}{Z(G)})'$ is cyclic.   We have $\frac{G}{Z(G)}$ has a normal subgroup $\frac{N}{Z(G)}$ of index $p$ and hence $(\frac{G}{Z(G)})' \subseteq \frac{N}{Z(G)}$. Consequently  $\mid (\frac{G}{Z(G)})' \mid=q, r$  or $qr$.

Now, if $\mid (\frac{G}{Z(G)})' \mid=q$, then $\frac{N}{Z(G)}$ is a cyclic normal subgroup of  $\frac{G}{Z(G)}$ of index $p$. But then by \cite[Lemma 4 (p. 303)]{zumud}, we have  $\mid \frac{G}{Z(G)}\mid \neq 1$, which is a contradiction.

Next suppose $\mid (\frac{G}{Z(G)})' \mid=r$. Now, if $(\frac{G}{Z(G)})' \subsetneq C(xZ(G))$ for some $xZ(G) \in (\frac{G}{Z(G)})' \setminus \lbrace 1 \rbrace$, then by Remark \ref{rem1},  $C(xZ(G))$ is an  abelian normal subgroup of $\frac{G}{Z(G)}$ of prime index. Consequently,  by \cite[Lemma 4 (p. 303)]{zumud} we have  $\mid \frac{G}{Z(G)}\mid \neq 1$, which is a contradiction

On the otherhand, suppose $(\frac{G}{Z(G)})' =C(xZ(G))$ for any $xZ(G) \in (\frac{G}{Z(G)})' \setminus \lbrace 1 \rbrace$. Then by \cite[Proposition 1.2.4]{elizabeth}, $\frac{G}{Z(G)}$ is a Frobenious group with Frobenious kernel $\frac{C(x)}{Z(G)}=C(xZ(G))=(\frac{G}{Z(G)})'$ and cyclic Frobenius complement $\langle yZ(G)\rangle = \frac{C(y)}{Z(G)}$ of order $pq$. Moreover, $C(y)$ is abelian by Remark \ref{rem1}. Therefore by \cite[Proposition 3.1]{amiri2019}, we have $\mid \Cent(G) \mid= r+2$.

Finally, If $\mid (\frac{G}{Z(G)})' \mid=qr$, then $\frac{G}{Z(G)} \cong C_{qr} {\rtimes}_\theta C_p$ and hence by \cite[Proposition 2.9]{baishya}, $\mid \Cent(G) \mid= qr+2$.

\end{proof}

A group $G$ is said to be primitive $n$-centralizer group if $\mid \Cent(G) \mid= \mid \Cent(\frac{G}{Z(G)}) \mid=n$. As an immediate consequence of the above results, we have the following:

\begin{prop}\label{lem1266}
If  $G$ is a finite group such that $\mid \frac{G}{Z(G)} \mid=pqr$,   $p<q<r$ being primes, then $G$ is primitive $n$-centralizer.  
\end{prop}

\begin{proof}
Combining Proposition \ref{lem19} and Theorem \ref{lem191} and noting that in the present situation  $\mid Z(\frac{G}{Z(G)})\mid =1$, we have $\mid \Cent(G) \mid=\mid (\frac{G}{Z(G)})'\mid+2=\mid \Cent(\frac{G}{Z(G)}) \mid$. 
\end{proof}

Now we consider the groups of order $p^2q$ for distinct primes $p$ and  $q$. Note that in our proofs $n_l$ and $S_l$ denotes the number of $l$-Sylow subgroups and any $l$-Sylow subgroup of a group respectively.

\begin{prop}\label{lem192}
Let $G$ be a non-abelian group and $p<q$ be primes.  
\begin{enumerate}
	\item  If $\mid G \mid =p^2q$, then $\mid \Cent(G) \mid= q+2$ or $G=A_4$.
	\item  If $\mid G \mid =pq^2$, then $\mid \Cent(G) \mid= q+2$ or $q^2+2$.
\end{enumerate}

\end{prop}

\begin{proof}
a) If $\mid G \mid =12$, then $G\cong A_4$, $T= \langle x, y \mid x^6=1, y^2=x^3, y^{-1}xy=x^{-1} \rangle$ or $D_{12}$.
It is easy to see that $\mid \Cent(A_4) \mid= 6$, $\mid \Cent(T) \mid= 5$ and $\mid \Cent(D_{12}) \mid= 5$.

Next suppose $\mid G \mid >12$. Clearly, we have $\mid Z(G) \mid =1$ or $p$. 

Now, if $\mid Z(G) \mid =p$, then $\mid\frac{G}{Z(G)} \mid=pq$ and hence by \cite[Corolary 2.5]{baishya}, we have $\mid \Cent(G) \mid= q+2$.

On the otherhand if $\mid Z(G) \mid =1$, then  for any Sylow subgroup $S$ of $G$ and any $x (\neq 1) \in S$, we have $C(x)=S$. Hence $\mid \Cent(G) \mid= n_p+n_q+1=q+2$.\\

b) Clearly, we have $\mid Z(G) \mid =1$ or $q$. 

Now, if $\mid Z(G) \mid =q$, then $\mid\frac{G}{Z(G)} \mid=pq$ and hence by \cite[Corolary 2.5]{baishya}, we have $\mid \Cent(G) \mid= q+2$.

On the otherhand if $\mid Z(G) \mid =1$, then  for any Sylow subgroup $S$ of $G$ and any $x (\neq 1) \in S$ we have $C(x)=S$. Hence $\mid \Cent(G) \mid= n_p+n_q+1=q^2+2$.

\end{proof}

The following lemma will be used in proving the next two results.

\begin{lem}\label{lem201}
Let $G$ be a finite CA-group. If $\frac{G}{Z(G)}$ is abelian, then $\frac{G}{Z(G)}$ is elementary abelian.
\end{lem}

\begin{proof}
We have $\lbrace \frac{C(x)}{Z(G)} / x \in G \setminus Z(G)\rbrace$ is a partition of $\frac{G}{Z(G)}$ (see \cite[Remark 2.1]{ed09}). Therefore by \cite[ (p. 571)]{zappa}, we have $\frac{G}{Z(G)}$ is elementary abelian.
\end{proof}

Let $p$ and  $q$ be two distinct primes. It may be mentioned here that in \cite[Theorem 1.2]{rashid}, the authors characterised non-abelian capable groups of order $p^2q$ using many technical and sophisticated tools. In the following result, we have characterised any capable group of order $p^2q$. We present a very elementary proof of this result using basic group theory.

\begin{rem}\label{rem151}
Let $p<q$ be primes. By  \cite{holder}, the only non-abelian group of order $p^2q$ such that $S_p=C_p \times C_p$ is $C_p \times (C_q \rtimes C_p)$.  
\end{rem}

\begin{thm}\label{lem393}
Let $G$ be any group and $p<q$ be primes. 
\begin{enumerate}
	\item  Suppose $\mid G \mid = pq^2$. Then $G$ is a capable group if and only if $\mid Z(G) \mid=1$. 
	\item  Suppose $\mid G \mid = p^2q$. Then $G$ is a capable group if and only if $\mid Z(G) \mid=1$ or $G \cong C_p \times (C_q \rtimes C_p)$.
\end{enumerate}

\end{thm}

\begin{proof}
By Remark \ref{rem1}, $G$ is a CA-group. Therefore in view of Lemma \ref{lem201}, $G$ is non-abelian. Now,\\
a) Suppose $G$ is a capable group of order $pq^2$. Then there exists a group $H$ such that $G \cong \frac{H}{Z(H)}$. By  Remark \ref{rem1}, we have both $G$ and $H$ are CA-groups. Now, suppose $\mid Z(\frac{H}{Z(H)}) \mid =q$.  Note that $\frac{H}{Z(H)}$ has exactly one centralizer of order $q^2$.  Let $C(sZ(H))$ and $C(tZ(H))$ be two distinct centralizers of $\frac{H}{Z(H)}$ of order $pq$. Then $C(sZ(H))=\langle uZ(H) \rangle =\frac{C(u)}{Z(H)}$ and $C(tZ(H))=\langle vZ(G) \rangle =\frac{C(v)}{Z(H)}$. It now follows that $\mid \frac{H}{C(u)}\mid = p= \mid \frac{H}{C(v)}\mid$. 
But then $\mid \frac{H}{Z(H)} \mid = \mid \frac{H}{{C(u) \cap C(v)}} \mid \leq p^2$, which is a contradiction.  Therefore $\mid Z(\frac{H}{Z(H)}) \mid =1$. Converse is trivial.\\

b) Suppose $G$ is a capable group of order $p^2q$. Then there exists a group $H$ such that $G \cong \frac{H}{Z(H)}$. Now, suppose $\mid Z(\frac{H}{Z(H)}) \mid =p$. Let $C(sZ(H))$ and $C(tZ(H))$ be two centralizers of $\frac{H}{Z(H)}$ of order $p^2$ and $pq$ respectively. If $S_p=C_{p^2}$, then $C(sZ(H))=\langle uZ(H) \rangle =\frac{C(u)}{Z(H)}$. Also we have  $C(tZ(H))=\langle vZ(G) \rangle =\frac{C(v)}{Z(H)}$. It now follows that $\mid \frac{H}{C(u)}\mid = q$ and $\mid \frac{H}{C(v)}\mid=p$. 
But then $\mid \frac{H}{Z(H)} \mid = \mid \frac{H}{{C(u) \cap C(v)}} \mid \leq pq$, which is a contradiction. Therefore $S_p = C_p \times C_p$ and hence $G \cong C_p \times (C_q \rtimes C_p)$  by Remark  \ref{rem151}.

Conversely, if $\mid Z(G) \mid=1$, then $G$ is capable.  On the otherhand, from \cite [Section 31]{western},  consider any group $H$ of order $p^3q$ given by $H=\langle a, b, c, d \mid a^p=b^p=c^p=d^q=1, ab=ba, ac=ca, ad=da, bd=db, c^{-1}bc=ab, c^{-1}dc=d^i \rangle$, where $q \equiv 1 \;(mod \;p)$ and $i^p \equiv 1 \;(mod \; q)$. Then $\frac{H}{Z(H)}$ is of order $p^2q$ with $p$-Sylow subgroup $S_p = C_p \times C_p$. Therefore  $ \frac{H}{Z(H)}  \cong C_p \times (C_q \rtimes C_p)$ by Remark  \ref{rem151}.

\end{proof}

Now we compute $\mid Cent(G) \mid$ for groups whose central factor is of order $p^2q$ for any primes $p, q$. Note that the group  $T= \langle x, y \mid x^6=1, y^2=x^3, y^{-1}xy=x^{-1} \rangle$ has non-trivial center and cyclic $2$-Sylow subgroups. Therefore $T$ is not capable by Theorem \ref{lem393}.

\begin{prop}\label{lem193}
Let $G$ be a finite group and $p<q$ be primes.  
\begin{enumerate}
	\item  If $\mid \frac{G}{Z(G)} \mid =12$, then $\mid \Cent(G) \mid= 6$ or $8$. Otherwise,
	\item  If $\mid \frac{G}{Z(G)} \mid =p^2q$, then $\mid \Cent(G) \mid= pq+2$ or $q+2$.
	\item  If $\mid \frac{G}{Z(G)} \mid =pq^2$, then $\mid \Cent(G) \mid= q^2+2$ or $q^2+q+2$.
\end{enumerate}

\end{prop}

 \begin{proof}
 In view of Remark \ref{rem1} and Lemma \ref{lem201}, $\frac{G}{Z(G)}$ is non-abelian. Now,
   
a) We have $\frac{G}{Z(G)} \cong A_4$ or $D_{12}$. If $\frac{G}{Z(G)} \cong A_4$, then $G$ has exactly $4$ centralizers of index $4$ and these will absorb exactly $\frac{2\mid G \mid}{3}$ non-central elements of $G$, noting that $G$ is a CA-group by  Remark \ref{rem1}. Let $k$ be the number of centralizers produced by the remaining $\frac{\mid G \mid}{3}$ elements of $G$. Now, if $G$ has a centralizer of index $3$, then $k=1$. Suppose $G$ has no centralizer of index $3$. Then $\frac{\mid G \mid}{3}=k(\frac{\mid G \mid}{6}-\frac{\mid G \mid}{12})+ \frac{\mid G \mid}{12}$ and consequently, $k=3$. Hence $\mid \Cent(G) \mid=6$ or $8$. Again, if $\frac{G}{Z(G)} \cong D_{12}$, then by \cite[Proposition 2.9]{baishya}, $\mid \Cent(G) \mid=8$.\\

b) Clearly $\mid Z(\frac{G}{Z(G)}) \mid =1$ or $p$. Now, If $\mid Z(\frac{G}{Z(G)}) \mid =p$, then by Theorem \ref{lem393},  $\frac{G}{Z(G)} \cong C_{pq} {\rtimes}_\theta  C_p$ and so $\mid \Cent(G) \mid= pq+2$ by \cite[Proposition 2.9]{baishya}.

Again if $\mid Z(\frac{G}{Z(G)}) \mid =1$, then we have $C(xZ(G))=S_q$ for any $xZ(G) (\neq 1) \in S_q$. Therefore  by \cite[Proposition 1.2.4]{elizabeth},  $\frac{G}{Z(G)}$ is a Frobenius group with Frobenius kernel $\frac{C(x)}{Z(G)}=C(xZ(G))$ and cyclic  Frobenius complement $\langle yZ(G)\rangle = \frac{C(y)}{Z(G)}$ of order $p^2$. Moreover by Remark \ref{rem1}, $C(y)$ is abelian . Therefore by \cite[Proposition 3.1]{amiri2019}, we have $\mid \Cent(G) \mid= q+2$. \\

c) By Theorem \ref{lem393}, we have $\mid Z(\frac{G}{Z(G)}) \mid =1$. Therefore $\frac{G}{Z(G)}$ is a Frobenius group with Frobenius kernel $\frac{K}{Z(G)}$ and  Frobenius complement $\frac{H}{Z(G)}$ of order $p$. Now, using \cite[Proposition 3.1]{amiri2019}, we have $\mid \Cent(G) \mid= q^2+2$ or $q^2+q+2$ (noting that if $Z(G)=Z(K)$, then $\mid \Cent(K) \mid=q+2$ by \cite[Corollary 2.5]{baishya}).
\end{proof}

Finally, we consider the  groups  of order $p^3$.

\begin{rem}\label{rem2}
Let $G$ be a group of order $p^3$. If $G$ is non-abelian then by \cite[Corollary 8.2]{capable},  $G$ is capable if and only if $G \cong  D_8$ or a $p$ group of exponent $p$.  It may be noted that $\frac{G}{Z(G)} \cong C_2 \times C_2 \times C_2$ for the group $G=Small Group (64, 60)$ in \cite{gap} (it is also pointed in \cite{ed09}). Again, if $p>2$, we have $\frac{G}{Z(G)} \cong C_p \times C_p \times C_p$ for the group $G$=Group $5.3.1$ in \cite[Section 3]{tayler}. Therefore,  if $G$ is abelian then in view of Lemma \ref{lem201}, $G$ is capable if and only if $G \cong C_p \times C_p \times C_p$. 
\end{rem}

It is easy to see that  for any non-abelian group $G$ of order $p^3$, we have $\mid \Cent(G) \mid= p+2$. We now compute  $\mid \Cent(G) \mid$ for any group $G$ (finite or infinite) whose central factor is of order $p^3$. This extends \cite[Theorem 3.3]{azad} and \cite[Proposition 2.10]{amirif4}. It may be mentioned here that  $\mid \Cent(G) \mid=\omega(G)+1$ if and only if $G$ is a CA-group (see arguments of \cite[Lemma 2.6]{ed09}), where $\omega(G)$ is the size of a  maximal set of pairwise non-commuting elements of $G$.

\begin{prop}\label{lem195}
Let $G$ be any group (finite or infinite) and  $p$ be a prime. If $\mid \frac{G}{Z(G)} \mid =p^3$, then $\mid \Cent(G) \mid= \omega(G)+1=p^2+p+2$ or $p^2+2$. 
\end{prop}

\begin{proof}

By Remark \ref{rem1}, $C(x) \cap C(y) = Z(G)$ for any $x,y \in G \setminus Z(G), xy \neq yx$. Now, suppose $G$ has no centralizer of index $p$. Then each proper centralizer of $G$ will contain exactly $p$ distinct right coset of $Z(G)$. Therefore $\mid \Cent(G) \mid= p^2+p+2$.

Next, suppose $G$ has a centralizer $C(x)$ of index $p$. Then $C(x)$ will contain exactly $p^2$ distinct right cosets of $Z(G)$. Therefore the number of right cosets of $Z(G)$ (other than $Z(G)$) left for the remaining centralizers is $p^3-p^2$. Clearly, in view of Second Isomorphism Theorem, $G$ cannot have another centralizer of index $p$ (noting that in the present scenario $C(x) \lhd G$ ) and consequently, any proper centralizer other than $C(x)$ will contain exactly $p$ distinct right cosets of $Z(G)$. Hence $\mid \Cent(G) \mid= p^2+2$.
\end{proof}



\end{document}